\documentclass[11pt,a4paper]{amsart}

\usepackage[T1]{fontenc}
\usepackage[utf8]{inputenc}
\usepackage{lmodern}
\usepackage{amsmath,amssymb,amsthm,mathtools}
\usepackage{mathrsfs}
\usepackage[noadjust]{cite}
\usepackage{enumitem,booktabs,longtable,array,tikz-cd}
\usepackage{xcolor}
\IfFileExists{microtype.sty}{\usepackage{microtype}}{}
\usepackage[
  colorlinks=true,
  linkcolor=red,
  citecolor=blue,
  urlcolor=blue
]{hyperref}
\usepackage{bookmark}

\allowdisplaybreaks[4]

\numberwithin{equation}{section}

\setlist[itemize]{topsep=4pt,itemsep=2pt,parsep=1pt}
\setlist[enumerate]{topsep=4pt,itemsep=3pt,parsep=1pt}

\theoremstyle{plain}
\newtheorem{theorem}{Theorem}[section]
\newtheorem{proposition}[theorem]{Proposition}
\newtheorem{lemma}[theorem]{Lemma}
\newtheorem{corollary}[theorem]{Corollary}

\theoremstyle{definition}

\newtheorem{definition}[theorem]{Definition}
\newtheorem{notation}[theorem]{Notation}
\newtheorem*{convention}{Convention}
\newtheorem{recollection}[theorem]{Recollection}

\theoremstyle{remark}
\newtheorem{remark}[theorem]{Remark}

\hypersetup{
  pdftitle={Tensor Geometry of Category Representations},
  pdfauthor={Peng Xu}
}

\theoremstyle{plain}

\theoremstyle{definition}
\newtheorem{example}[theorem]{Example}

\newcommand{\op}{\mathrm{op}}
\DeclareMathOperator{\Spc}{Spc}
\DeclareMathOperator{\Spec}{Spec}
\DeclareMathOperator{\Perf}{Perf}

\DeclareMathOperator{\Loc}{Loc}

\DeclareMathOperator{\supp}{supp}

\newcommand{\RHom}{\operatorname{RHom}}

\makeatletter
\def\l@part{\@tocline{-1}{6pt plus1pt}{0pt}{}{\bfseries}}
\makeatother

\begin{document}
\ifdefined\synctex
  \synctex=1\relax
\fi
\hypersetup{pdftitle={Tensor Triangular Geometry of BAND Algebras }}
\markboth{\rightline{\scriptsize Peng Xu}}
         {\leftline{\scriptsize TENSOR TRIANGULAR GEOMETRY OF BAND ALGEBRAS}}
\begin{center}
{\large\bfseries TENSOR TRIANGULAR GEOMETRY OF BAND ALGEBRAS\\[3pt]
}
\vspace{0.5cm}

Peng Xu\footnote{Corresponding author. Peng Xu is supported by the
National Natural Science Foundation of China (Grant No.~12231009,
11971224).}
\end{center}
\bigskip
\centerline{\bfseries Abstract}
\medskip
\begingroup
\leftskip10truemm
\rightskip10truemm
\noindent\hspace{1em}
In this paper, we study the tensor triangular geometry of band algebras $RB$, where $B$ is a finite left regular band and $R$ is a commutative Noetherian ring.  We show that the Balmer spectrum of $\Perf(RB)$ is homeomorphic to $\Spec R\times L_B$, where $L_B$ is the support lattice of $B$. We also show that there is a bijection between the localizing  ideals of  $D(RB)$ and the stable subsets of $\Spec R\times L_B$. As a consequence, the telescope conjecture holds for $D(RB)$. Finally, we construct a band $B$ such that the generalized  Nerves of Steel Conjecture fails for the non-rigid tensor triangulated category $\Perf(RB)$.\\[2mm]

\par
\endgroup
\bigskip

\section{Introduction}
Tensor triangular geometry \cite{Balmer,BalmerICM} associates to a tensor triangulated category $\mathcal K$ a topological space $\Spc(\mathcal K)$, called the Balmer spectrum, with the aim of classifying thick  ideals and localizing ideals via support theory. In recent years it has attracted attention  from diverse fields, including algebraic geometry \cite{Aoki23,BG22Real,BG25Perm,BG25Artin,Gal18,Gal19,Neeman}, modular representation theory \cite{BIK11,BIK11mod,BIKP18,BBIKP25}, and algebraic topology \cite{BHNNNS19,BBPSW25,BS17}.

In this paper, our goal is to study the tensor triangular geometry of the derived category of the band algebra $RB$.  Here $R$ is a commutative Noetherian ring and $B$ is a finite left regular band, that is, a finite monoid $B$ such that \( b^2=b\) and \(bcb=bc\) for all \( b,c\in B.\) Band algebras arise naturally from combinatorial settings, such as matroids, oriented matroids, finite simple graphs, and finite acyclic quivers, see \cite{MSS}. In particular, one can associate to a finite acyclic quiver $Q$ a finite left regular band $B_Q^\prec$ such that  $RB_Q^\prec=RQ$. In \cite{Sabatini}, Sabatini has  made a systematic study of the tensor triangular geometry of $D(RQ)$ with the vertexwise tensor product. There is another tensor structure on $D(RQ)$, induced by the band structure of $B_Q^\prec$. Moreover, these two tensor structures on the  derived category $D(RQ)$ can have non-homeomorphic  Balmer spectra. The following is our first main result:

\vspace{2mm}
\noindent\textbf{Theorem 1.1} (Theorem~\ref{thm6.7}). Let $R$ be a commutative Noetherian ring and $Q=(Q_0,Q_1)$ a finite acyclic quiver. Then the  triangulated category $\Perf(RQ)$ carries tensor products with spectra
\begin{align*}
 \Spc(\Perf(RQ),\otimes_{\mathrm{band}}^{\mathbf L})
   &\cong\Spec R\times(Q_0,\prec),\\
 \Spc(\Perf(RQ),\otimes_{\mathrm{vert}}^{\mathbf L})
   &\cong\Spec R\times Q_0^{\mathrm{disc}}.
\end{align*}
Here,  $(Q_0,\prec)$ is the vertex set of $Q$ equipped with a total order satisfying $s(\alpha)\prec t(\alpha)$ for every arrow $\alpha$. It carries the topology whose closed subsets are precisely the downward-closed subsets. More generally, for each finite left regular band $B$, one has a support lattice \(L_B=\{Bb:b\in B\}\) which is ordered
by reverse inclusion. We always endow $L_B$  with the topology whose closed subsets are precisely the downward-closed subsets. This space determines the Balmer spectrum of $\Perf(RB)$.

\vspace{2mm}
\noindent\textbf{Theorem 1.2} (Theorem~\ref{thm3.7}). Let $R$ be a commutative Noetherian ring and $B$ a finite left regular band. Then there is a homeomorphism
$\Spec R\times L_B\xrightarrow{\sim}\Spc\Perf(RB)$.

\vspace{2mm}
For a rigidly-compactly generated tensor triangulated category $\mathcal J$, stratification provides a powerful framework for studying localizing tensor ideals \cite{BHS23,Ste13,BIK11}. Less is known when $\mathcal J^c$ is not rigid. In this setting, the classification of localizing tensor ideals has been studied for derived categories of global representations associated with certain families of finite groups \cite{XuTelescope}, as well as for derived categories of noetherian path algebras \cite{Sabatini}. For a commutative Noetherian ring $R$ and a nontrivial finite left regular band $B$, the derived category $D(RB)$ is compactly generated while its compact subcategory $\Perf(RB)$ is not rigid. We introduce  tensor support (see Definition~\ref{def3.1}) and  stable subsets of $\Spec R\times L_B$ (see Definition~\ref{def3.5}).  Then we show that the localizing ideals of $D(RB)$ are precisely controlled by the stable subsets of $\Spec R\times L_B$:

\vspace{2mm}
\noindent\textbf{Theorem 1.3} (Theorem~\ref{thm3.6}). For a commutative Noetherian ring $R$ and a finite left regular band $B$, there are mutually inverse bijections
\[
\left\{\text{localizing ideals of }D(RB)\right\}
\underset{\Psi}{\overset{\Phi}{\rightleftarrows}}
\left\{W\subseteq\Spec R\times L_B:W\text{ is stable}\right\},
\]
given by
\(
\Phi(\mathcal J)=\bigcup_{X\in\mathcal J}\operatorname{Supp}_B(X)\), \(\Psi(W)=\{X:\operatorname{Supp}_B(X)\subseteq W\}=\Loc\langle\kappa(\mathfrak p)_F:(\mathfrak p,F)\in W\rangle.\)
\vspace{2mm}

Under suitable hypotheses on its Balmer spectrum, stratification of a rigidly-compactly generated tensor triangulated category $\mathcal J$ implies the telescope conjecture, i.e., every smashing ideal of $\mathcal J$ is compactly generated as a localizing subcategory. This conjecture can be generalized to the non-rigid setting by asking for the same property in compactly generated tensor triangulated categories. We establish this for $D(RB)$:

\vspace{2mm}
\noindent\textbf{Theorem 1.4} (Theorem~\ref{thm5.3}). Let $R$ be a commutative Noetherian ring and $B$ a finite left regular band. Suppose $W\subseteq\Spec R\times L_B$ is stable. Then the following  are equivalent:
\begin{enumerate}[label=\textup{(\roman*)}]
\item $\Psi(W)$ is smashing;
\item $W$ is specialization-closed in $\Spec R\times L_B$;
\item $\Psi(W)$ is generated by compact objects
      as a localizing subcategory.
\end{enumerate}
In particular, $D(RB)$ satisfies the telescope conjecture.
\vspace{2mm}

Balmer's Nerves of Steel Conjecture states that, for a rigid tensor triangulated category $\mathcal{K}$, the natural comparison map \(\Spc^{\mathrm h}(\mathcal{K})\rightarrow \Spc( \mathcal{K})\) is always a homeomorphism. It is possible to generalize this to the non-rigid setting. While the counterexamples to this conjecture have been given in \cite{BHR26, Coulembier26, Hyslop}, we construct a band $B$ such that the generalized Nerves of Steel Conjecture fails for $\Perf(RB)$, which might be of independent interest.

\vspace{2mm}

\noindent\textbf{Structure of the paper.} In Section~\ref{sec2}, we recall some notions from tensor triangular geometry and present some basic properties of the derived categories of band algebras. In Section~\ref{sec3}, we compute the Balmer spectrum of $\Perf(RB)$ and classify all localizing ideals of $D(RB)$. We also construct a band algebra for which the generalized Nerves of Steel Conjecture fails. In Section~\ref{sec4}, we prove that the tensor abelian spectrum of $\mathrm{mod}(RB)$ coincides with $\Spc\Perf(RB)$.  In Section~\ref{sec5}, we prove that the
telescope conjecture holds for $D(RB)$. Finally, in Section~\ref{sec6}, we study several examples of band algebras arising from various combinatorial settings.

\vspace{2mm}

\begin{convention}Throughout the paper every monoid is finite and $R$ is a  commutative Noetherian ring. All modules are unital left modules. For a ring $A$, the symbol $D(A)$ denotes the unbounded derived category of left $A$-modules, and $\Perf(A)$ its full subcategory of objects isomorphic to bounded complexes of finitely generated
projective modules. For $C\in D(R)$, write $\supp_R(C)$ for its Balmer--Favi support.
\end{convention}

\vspace{2mm}

\noindent\textbf{Declaration of generative AI use.}
The author used  OpenAI's GPT6 as an exploratory and editorial aid in formulating several of the questions, exploring possible approaches, and assisting with typesetting
and manuscript preparation. The mathematical content, verification of the arguments, and all final decisions concerning the mathematics and its presentation are the author's own.

\section{Preliminaries}\label{sec2}
In this section, we recall some notions from tensor triangular geometry to fix our notation. We then present some basic properties of the derived categories of band algebras.
\subsection{Tensor triangulated categories}
We assume some familiarity with the basic notions of tensor triangular geometry. See for example \cite{Balmer, BalmerICM}.
\begin{recollection} Let $\mathcal{J}$ be a compactly generated triangulated category with compact part $\mathcal{J}^{c}$. A thick subcategory $\mathcal{S}$ is called localizing if it is closed under coproducts.  A localizing subcategory $\mathcal{S}$ is called a Bousfield subcategory if the Verdier quotient $q:\mathcal{J}\rightarrow \mathcal{J}/\mathcal{S}$ has a right adjoint. Moreover, if $\mathcal{S}^{\bot}$ is localizing, $\mathcal{S}$ is called a smashing subcategory. Here $\mathcal{S}^{\bot}:=\{Y\in \mathcal{J}~|~\mathrm{Hom}_{\mathcal{J}}(X,Y)=0, \forall X\in \mathcal{S}\}$.\\

A compactly generated tensor triangulated category $(\mathcal{J},\otimes,1)$ is a compactly generated triangulated category $\mathcal{J}$ with a tensor product $\otimes:\mathcal{J}\times \mathcal{J}\rightarrow \mathcal{J}$ which is a triangulated functor in each variable and also makes $\mathcal{J}$ symmetric monoidal with unit $1$. We require $\otimes$ to preserve coproducts. A subcategory $\mathcal{S}$ is called a tensor ideal if $X\otimes Y\in \mathcal{S}$ for $X\in \mathcal{S}$, $Y\in \mathcal{J}$. A localizing subcategory is called a localizing  ideal if it is a tensor ideal. Similarly, a smashing subcategory is called a smashing  ideal if it is a tensor ideal. Let $\mathcal{C}$ be a collection of objects in $\mathcal{J}$. We denote by $\mathrm{Loc}\langle\mathcal{C}\rangle$ the localizing subcategory generated by $\mathcal{C}$, and by $\mathrm{Loc}_{\otimes}\langle\mathcal{C}\rangle$ the localizing ideal generated by $\mathcal{C}$. If $\mathcal{C}$ is a collection of compact objects, we denote by $\mathrm{thick}\langle\mathcal{C}\rangle$ the thick subcategory generated by $\mathcal{C}$, and by $\mathrm{thick}_{\otimes}\langle\mathcal{C}\rangle$ the thick ideal generated by $\mathcal{C}$.
\end{recollection}

\begin{definition}[{\cite[Definition~2.15]{XuTelescope}}] Let $\mathcal{J}$ be a compactly generated tensor triangulated category. We say that $\mathcal{J}$ satisfies the telescope conjecture if every smashing ideal of $\mathcal{J}$ is compactly generated.
\end{definition}

\begin{remark}Our definition of the telescope conjecture coincides with \cite[Definition 9.1]{BHS23} when $\mathcal{J}$ is  a rigidly-compactly generated tensor triangulated category in the sense of \cite{BF11}. In the non-rigid setting, our formulation should be distinguished from Liang's generalized telescope conjecture \cite[Definition~4.8]{Liang}.
\end{remark}

\subsection{Band algebras}
Left regular bands arise naturally in algebraic combinatorics \cite{MSS}. Throughout this paper, all left regular bands are assumed to be finite. From now on, let $R$ be a commutative Noetherian ring. We will study the derived categories of band algebras over $R$.

\begin{recollection}\label{rec2.4}
A \emph{  left regular band} is a  monoid $B$ satisfying \( b^2=b\) and \(bcb=bc\) for all \( b,c\in B.\) Let $B$ be a  left regular band. Its \emph{support lattice}
is the finite set \(L_B=\{Bb:b\in B\},\) which is ordered by reverse inclusion. The \emph{support map}
\[
\sigma:B\longrightarrow L_B,\qquad b\longmapsto Bb
\]
is surjective, and $\sigma(b)$ is called the \emph{support} of $b$.
For all $b,c\in B$, one has
\[
\sigma(bc)=\sigma(b)\vee\sigma(c),\qquad
bc=b\ \Longleftrightarrow\ \sigma(c)\leq\sigma(b).
\]
It follows that \(\sigma(b)=\sigma(c)\) if and only if \(bc=b\ \text{and}\ cb=c.\) See  \cite[Section~2.2]{MSS} for more details.
\end{recollection}

Write $RB$ for the monoid algebra of $B$ over $R$, which we call the band algebra of $B$ over $R$. It is the free $R$-module
\[
 RB=\bigoplus_{b\in B}R\,b,
\]
with operations
\[
 \begin{aligned}
 \left(\sum_{b\in B}r_b b\right)
 +\left(\sum_{b\in B}s_b b\right)
 &=\sum_{b\in B}(r_b+s_b)b,\\
 \left(\sum_{b\in B}r_b b\right)
 \left(\sum_{c\in B}s_c c\right)
 &=\sum_{b,c\in B}r_b s_c\,(bc).
 \end{aligned}
\]
Tensor products of $RB$-modules are taken over $R$, with diagonal $B$-action. The tensor unit is $R$ with the trivial $B$-action. Since $RB$ is finite free over
$R$, it is left Noetherian, and an $RB$-module is finitely generated over $RB$ if and only if it is finitely generated over $R$.

\begin{notation}
 For $F\in L_B$, write $\mathord\downarrow F=\{G:G\leq F\}$ and $\mathord\uparrow F=\{G:F\leq G\}$. We equip $L_B$ with the topology whose closed subsets are
precisely the downward-closed subsets.
\end{notation}

\begin{remark} For $(\mathfrak p,F)\in\Spec R\times L_{B}$, we have \(\overline{\{(\mathfrak p,F)\}}=V(\mathfrak p)\times\mathord\downarrow F.\)
Every closed subset $Z$ of $\Spec R\times L_{B}$ is of the form
$Z=\bigcup_{F\in L_B} C_F\times\mathord\downarrow F$ where \(C_F:=\{\mathfrak p\in\Spec R:(\mathfrak p,F)\in Z\}.\) In particular, $\Spec R\times L_{B}$ is Noetherian.
\end{remark}

\begin{notation}For $F\in L_B$, let $R_F$ be the left $RB$-module whose underlying $R$-module is $R$, with $B$-action
\[
 b\cdot r=\chi_F(b)r,\qquad\text{where}\quad
 \chi_F(b)=
 \begin{cases}
  1,&\sigma(b)\leq F,\\
  0,&\text{otherwise}.
 \end{cases}
\]
In particular, $R=R_{\max L_B}$ is the tensor unit, and $R_F\otimes_R R_G=R_{F\wedge G}$.
\end{notation}

\begin{lemma}\label{lem2.8}
Let B be a left regular band. Then there are idempotents $(e_F)_{F\in L_B}$ in $RB$ with the following properties:
\begin{enumerate}
\item They are pairwise orthogonal and satisfy
      $\sum_{F\in L_B}e_F=1$.
\item For $F,G\in L_B$, one has
      $e_G(RB)e_F=0$ unless $G\leq F$.
\item For $F\in L_B$, one has $e_F(RB)e_F=Re_F$.
\item For $F,G\in L_B$, one has
      $\chi_G(e_F)=\delta_{F,G}$.
\end{enumerate}
\end{lemma}

\begin{proof}
For each $F\in L_B$, fix $b_F\in B$ with $\sigma(b_F)=F$. Define $e_F\in RB$ recursively by
\[
e_F=b_F\left(1-\sum_{\substack{G\in L_B\\G>F}}e_G\right),
\]
where all $e_G$ with $G>F$ are defined before $e_F$. In particular, the empty sum gives $e_{\max L_B}=b_{\max L_B}$. Then these assertions follow from \cite[Theorem 5.1]{MSS} and the proof of \cite[Theorem 5.3]{MSS}.
\end{proof}

\begin{remark}\label{rem2.9}
For $F\in L_B$, the functor
\[
e_F(-):\operatorname{Mod}(RB)\longrightarrow\operatorname{Mod}(R),
\qquad M\longmapsto e_FM,
\]
is exact and preserves coproducts. Indeed, it is naturally isomorphic to $\operatorname{Hom}_{RB}((RB)e_F,-)$, and $(RB)e_F$ is finitely generated projective over $RB$. For any commutative $R$-algebra $S$, let $e_F^S$ denote the image of $e_F$ in $SB$. The split decomposition $M=e_FM\oplus(1-e_F)M$ gives a natural isomorphism
\(S\otimes_R e_FM\cong e_F^S(S\otimes_R M).\) Thus $e_F(-)$ commutes with extension of scalars. Applying $e_F(-)$ degreewise defines a functor on complexes, which we denote by the same symbol. This functor has the same properties.
\end{remark}

\begin{proposition}\label{prop2.10}
Let B be a left regular band and $F\in L_B$. Then $R_F$ admits a bounded resolution by finitely generated projective $RB$-modules.
\end{proposition}

\begin{proof}
Let $(e_F)_{F\in L_B}$ be as in Lemma~\ref{lem2.8}. For $F,G\in L_B$, the module $e_G(RB)e_F$ is the image of the idempotent $R$-linear endomorphism
\[
 RB\longrightarrow RB,\qquad x\longmapsto e_Gxe_F.
\]
It is therefore a direct summand of the finite free $R$-module $RB$, hence finitely generated projective over $R$. Enumerate the elements of $L_B$ as $F_1,\ldots,F_m$ so that
each element appears before every strictly smaller element. For an $RB$-module $M$, the submodules
\[
 M_{\geq i}=\bigoplus_{j\geq i}e_{F_j}M
\]
are $RB$-submodules of $M$ by Lemma~\ref{lem2.8}. By the same lemma we also have \(e_{F_i} a e_{F_i}=\chi_{F_i}(a)e_{F_i}\) for $a\in RB$.
Hence $M_{\geq i}/M_{\geq i+1}\cong R_{F_i}\otimes_R e_{F_i}M$.

The $R$-linear extension of $\chi_{F_i}$ to $RB$ restricts to an epimorphism of $RB$-modules \((RB)e_{F_i}\twoheadrightarrow R_{F_i}.\) Its restriction to $e_{F_i}(RB)e_{F_i}=Re_{F_i}$ is an isomorphism of $R$-modules. We proceed by descending induction on $i$. The case $i=m$ follows from $R_{F_m}\cong(RB)e_{F_m}$.
Assume that the assertion holds for all $j>i$. Let $K_i=\ker((RB)e_{F_i}\to R_{F_i})$. Then $e_{F_j}K_i=0$ for $j\leq i$.
For $j>i$, the identity $\chi_{F_i}(e_{F_j})=0$ implies \(e_{F_j}K_i=e_{F_j}(RB)e_{F_i}.\) Applying the filtration above to $K_i$, its successive quotients are therefore
\[
R_{F_j}\otimes_R e_{F_j}K_i
\cong R_{F_j}\otimes_R e_{F_j}(RB)e_{F_i},
\qquad j>i.
\]
Since $e_{F_j}(RB)e_{F_i}$ is finitely generated projective over $R$, the induction hypothesis gives a bounded resolution of each quotient by finitely generated projective $RB$-modules. Successive applications of the horseshoe lemma give such a resolution of $K_i$. The short exact sequence
\[
0\longrightarrow K_i\longrightarrow(RB)e_{F_i}
\longrightarrow R_{F_i}\longrightarrow0,
\]
together with this resolution and the projectivity of $(RB)e_{F_i}$, gives the required bounded resolution of $R_{F_i}$. This completes the induction.
\end{proof}

\begin{lemma}\label{lem2.11}
Enumerate $L_B$ as $F_1,\ldots,F_m$ so that each element appears before every strictly smaller element. Then every complex $X$ of $RB$-modules has a functorial finite
filtration with successive quotients $R_{F_i}\otimes_R e_{F_i}X$ for $1\leq i\leq m$.
\end{lemma}

\begin{proof}
It suffices to apply the filtration constructed in the proof of Proposition~\ref{prop2.10} degreewise.
\end{proof}

\begin{proposition}\label{prop2.12}
Let B be a left regular band. Then a complex $X$ of $RB$-modules is perfect over $RB$ if and only if its underlying $R$-complex is perfect, or equivalently, if and only if $e_FX$ is perfect over $R$ for every $F\in L_B$.
\end{proposition}

\begin{proof}
Since $RB$ is finite free over $R$, every finitely generated projective $RB$-module is finitely generated projective over $R$. Thus perfection over $RB$ implies perfection
of the underlying $R$-complex. The decomposition of $R$-complexes \(X=\bigoplus_{F\in L_B}e_FX\) shows that the underlying $R$-complex is perfect if and only if every $e_FX$ is perfect over $R$.

Conversely, suppose that every $e_FX$ is perfect over $R$. By Proposition~\ref{prop2.10}, each $R_F$ admits a bounded resolution by finitely generated projective $RB$-modules. It follows that $R_F\otimes_R e_FX$ is perfect over $RB$. The finite filtration in Lemma~\ref{lem2.11} then implies that $X$ is perfect over $RB$.
\end{proof}

\begin{corollary}\label{cor2.13}
Let B be a left regular band. Then the category $D(RB)$ is compactly generated with compact part $\Perf(RB)$, and the diagonal derived tensor product $-\otimes_R^{\mathbf L}-$ preserves compact objects.
\end{corollary}

\begin{remark}\label{rem2.14}
Unless $B$ is the trivial monoid, $D(RB)$ is not rigidly-compactly generated. Indeed, one can choose $F<\max L$ and $b$ with $\sigma(b)\nleq F$. Then the compact object $R_F$ is not dualizable. If it had a dual $D$, the underlying $R$-linear morphism of the coevaluation $R\to R_F\otimes_R^{\mathbf L}D$ would be zero. Applying the forgetful functor to the triangle identity would then imply that $R_F=0$, a contradiction.
\end{remark}

\section{The Balmer spectrum}\label{sec3}
In this section, we show that the Balmer spectrum of $\Perf(RB)$ is $\Spec R\times L_B$, where $L_B$ is the support lattice of $B$. We also show that there is a bijection between the localizing  ideals of  $D(RB)$ and the stable subsets of $\Spec R\times L_B$. At the end of this section, we construct a band algebra such that  the generalization of Balmer's Nerves of Steel Conjecture to non-rigid setting fails for $\Perf(RB)$.

Fix a commutative Noetherian ring $R$ and a (finite) left regular band $B$.  For $F\in L_B$, choose $b_F\in B$ with $\sigma(b_F)=F$.  Then there is an exact functor
\[
\omega_F:\operatorname{Mod}(RB)\longrightarrow\operatorname{Mod}(R),
\qquad M\longmapsto b_FM.
\]
 In particular, we have \(\omega_F(R_G)=R\) if \(F\leq G\) and \(\omega_F(R_G)=0\) otherwise. Applying this functor degreewise induces an exact functor
\[
\omega_F:D(RB)\longrightarrow D(R),
\qquad X\longmapsto b_FX,
\]
which we denote by the same symbol.  It is strong symmetric monoidal because \(b_F(M\otimes_RN)=(b_FM)\otimes_R(b_FN)\) and \(b_FR=R.\)

\begin{definition}\label{def3.1}
For $X\in D(RB)$, define the tensor support of $X$ by
\[
 \operatorname{Supp}_B(X):=
 \{(\mathfrak p,F)\in\Spec R\times L:
 \kappa(\mathfrak p)\otimes_R^{\mathbf L}\omega_F(X)\ne0\}.
\]
\end{definition}

\begin{lemma}\label{lem3.2}
For $X\in D(RB)$, the following hold:
\begin{enumerate}
\item One has
\(
\operatorname{Supp}_B(X)=
\bigcup_{G\in L_B}\supp_R(e_GX)\times\mathord\downarrow G,\) which is independent of the choice of $b_F$.

\item One has $\operatorname{Supp}_B(X)=\varnothing$
if and only if $X\simeq0$.

\item For every $Y\in D(RB)$, one has
\(\operatorname{Supp}_B(X\otimes_R^{\mathbf L}Y)=
\operatorname{Supp}_B(X)\cap\operatorname{Supp}_B(Y).\)

\item If $X$ is compact, then $\operatorname{Supp}_B(X)$
is closed in $\Spec R\times L_B$.
\item For any family $(X_i)_{i\in I}$ in $D(RB)$, one has
\(
\operatorname{Supp}_B\left(\bigoplus_{i\in I}X_i\right)
=\bigcup_{i\in I}\operatorname{Supp}_B(X_i).
\)

\item \(\operatorname{Supp}_B(Y)\subseteq\operatorname{Supp}_B(X)\cup\operatorname{Supp}_B(Z)\), for any  triangle \(X\longrightarrow Y\longrightarrow Z\longrightarrow X[1]\).
\end{enumerate}
\end{lemma}

\begin{proof} For (1), if $b,c\in B$ have support $F$, then $bc=b$ and $cb=c$. Multiplication by $b$ and $c$ gives inverse natural isomorphisms between $cX$ and $bX$.
Hence $\supp_R(bX)=\supp_R(cX)$. Thus the tensor support is independent of the choice of $b_F$.   For each $n\in\mathbb Z$, apply the exact functor $b_F(-)$
to the finite filtration of $H^n(\kappa(\mathfrak p)\otimes_R^{\mathbf L}X)$ constructed in the proof of Proposition~\ref{prop2.10}.
On the quotient $\kappa(\mathfrak p)_G\otimes_{\kappa(\mathfrak p)} e_GH^n(\kappa(\mathfrak p)\otimes_R^{\mathbf L}X)$,
$b_F$ acts as the identity if $F\leq G$ and as zero otherwise. Then exactness of $b_F(-)$ and $e_G(-)$ gives
\[
\kappa(\mathfrak p)\otimes_R^{\mathbf L}\omega_F(X)\ne0
\quad\Longleftrightarrow\quad
\kappa(\mathfrak p)\otimes_R^{\mathbf L}e_GX\ne0
\text{ for some }G\geq F.
\]
This proves the formula.

By (1), $\operatorname{Supp}_B(X)=\varnothing$ if and only if $\supp_R(e_GX)=\varnothing$ for every $G$. Since support detects zero objects in $D(R)$ and
\(X=\bigoplus_{G\in L_B}e_GX\) as $R$-complexes, it follows that $\operatorname{Supp}_B(X)=\varnothing$ if and only if $X\simeq0$. This proves (2).

For (3), the functor $\kappa(\mathfrak p)\otimes_R^{\mathbf L}\omega_F(-)$ is strong symmetric monoidal with values in $D(\kappa(\mathfrak p))$. Since the derived tensor product of two nonzero complexes over a field is nonzero, then the desired identity follows.

For (4), if $X$ is compact, every $e_GX$ is perfect over $R$, so $\supp_R(e_GX)$ is closed in $\Spec R$. Since each $\mathord\downarrow G$ is closed in $L_B$, the formula in (1) expresses $\operatorname{Supp}_B(X)$ as a finite union of closed subsets of $\Spec R\times L_B$, proving the assertion.

The proofs of (5) and (6) are routine.
\end{proof}

\begin{lemma}\label{lem3.3}
Let $k$ be a field and $X\in D(kB)$. Then
\[
 \Loc_\otimes\langle X\rangle=\Loc\langle k_F:\omega_F(X)\ne0\rangle.
\]
Moreover, for every $F\in L_B$, one has \(k_F\in\Loc\langle X\otimes_k k_G:G\in L_B\rangle\) if and only if \(\omega_F(X)\ne0.\)
\end{lemma}

\begin{proof}

By Lemma~\ref{lem2.11}, there is a finite functorial filtration of $X$ with quotients $k_G\otimes_k e_GX$. Because $\omega_F$ and $e_G(-)$ are exact, the filtration of every $H^n(X)$  implies that $\omega_F(X)\ne0$ if and only if there exists some $G$ such that $F\leq G$ and $H^*(e_GX)\ne0$. Since $e_GX$ is a complex
of vector spaces, it is isomorphic in $D(k)$ to the direct sum of its cohomology groups, and so
\[
 X\in\Loc\langle k_G:H^*(e_GX)\ne0\rangle\subseteq \Loc\langle k_F:\omega_F(X)\ne0\rangle.
\]
Since the set $\{F: \omega_F(X)\ne0\}$ is downward-closed, $\Loc\langle k_F:\omega_F(X)\ne0\rangle$ is a localizing ideal. This proves the \(\subseteq\) containment.

For the reverse containment, we prove by induction. Fix $F$ and assume that $k_H\in\Loc_\otimes\langle X\rangle$ for every $H<F$ with $ \omega_H(X)\ne0$.
Tensoring the filtration of $X$ with $k_F$ and using $k_F\otimes_k k_G\cong k_{F\wedge G}$ shows that $e_H(k_F\otimes_k X)=0$ unless $H\leq F$.
Since $\omega_F(k_F\otimes_k X)\cong\omega_F(X)\ne0$, we have $H^*(e_F(k_F\otimes_k X))\ne0$. For every $H<F$, the induction hypothesis gives
\[
k_H\otimes_k e_H(k_F\otimes_k X)
\in\Loc\langle k_H\rangle
\subseteq\Loc_\otimes\langle X\rangle.
\]
Thus all successive quotients in the filtration of $k_F\otimes_k X$, except possibly the quotient indexed by $F$, belong to $\Loc_\otimes\langle X\rangle$.
Since $k_F\otimes_k X\in\Loc_\otimes\langle X\rangle$, the triangles associated with this filtration imply that
\[
k_F\otimes_k e_F(k_F\otimes_k X)
\in\Loc_\otimes\langle X\rangle.
\]
As $e_F(k_F\otimes_k X)$ has nonzero cohomology, some shift of $k_F$ is a direct summand of this quotient. Hence $k_F\in\Loc_\otimes\langle X\rangle$, completing the induction.

For the final assertion, the forward implication follows by the same induction, and the reverse implication follows by applying $\omega_F$.
\end{proof}

\begin{lemma}\label{lem3.4}
For every $X\in D(RB)$, one has
\[
 X\in\Loc\langle X\otimes_R^{\mathbf L}\kappa(\mathfrak p):
                          \mathfrak p\in\Spec R\rangle.
\]
\end{lemma}

\begin{proof}
Neeman's classification for a commutative Noetherian ring implies
\[D(R)=\Loc\langle\kappa(\mathfrak p):\mathfrak p\in\Spec R\rangle,\] see \cite{Neeman}. The conclusion follows by applying the exact coproduct-preserving functor
$X\otimes_R^{\mathbf L}-\colon D(R)\to D(RB)$ to \(R\in\Loc\langle\kappa(\mathfrak p):\mathfrak p\in\Spec R\rangle.\) Here the second factor carries the trivial $B$-action, so its value at $R$ is $X$.
\end{proof}

\begin{definition}\label{def3.5}
A subset $W\subseteq\Spec R\times L_B$ is \emph{stable} if $(\mathfrak p,F)\in W$ and $G\leq F$ imply $(\mathfrak p,G)\in W$.
\end{definition}

\begin{theorem}\label{thm3.6}
For a commutative Noetherian ring $R$ and a finite left regular band $B$, there are mutually inverse bijections
\[
\left\{\text{localizing ideals of }D(RB)\right\}
\underset{\Psi}{\overset{\Phi}{\rightleftarrows}}
\left\{W\subseteq\Spec R\times L_B:W\text{ is stable}\right\},
\]
given by
\[
\Phi(\mathcal J)=\bigcup_{X\in\mathcal J}\operatorname{Supp}_B(X),\qquad \Psi(W)=\{X:\operatorname{Supp}_B(X)\subseteq W\}=\Loc\langle\kappa(\mathfrak p)_F:(\mathfrak p,F)\in W\rangle.
\]
\end{theorem}

\begin{proof}
By Lemma~\ref{lem3.2}, $\Phi(\mathcal J)$ is stable and $\Psi(W)=\{X:\operatorname{Supp}_B(X)\subseteq W\}$ is a localizing ideal, so both maps are well defined. We first prove $\{X:\operatorname{Supp}_B(X)\subseteq W\}=\Loc\langle\kappa(\mathfrak p)_F:(\mathfrak p,F)\in W\rangle$ for a stable set $W$. The $\supseteq$ containment follows from the equality
\(\operatorname{Supp}_B(\kappa(\mathfrak p)_F)=\{\mathfrak p\}\times\mathord\downarrow F.\) Conversely, suppose $\operatorname{Supp}_B(X)\subseteq W$.
For each $\mathfrak p\in\Spec R$, Lemma~\ref{lem3.3}, together with the compatibility of $\omega_F$ with base change, gives
\[
X\otimes_R^{\mathbf L}\kappa(\mathfrak p)
\in\Loc\langle\kappa(\mathfrak p)_F:
(\mathfrak p,F)\in\operatorname{Supp}_B(X)\rangle.
\]
Since $\operatorname{Supp}_B(X)\subseteq W$, Lemma~\ref{lem3.4} now implies
\[
X\in\Loc\langle\kappa(\mathfrak p)_F:
(\mathfrak p,F)\in W\rangle.
\]
This proves the equality. The same argument also implies that $\Phi(\Psi(W))=W$.

It remains to prove $\Psi(\Phi(\mathcal J))=\mathcal J$. The inclusion $\mathcal J\subseteq\Psi(\Phi(\mathcal J))$ follows from the definitions. Conversely, let $X\in\Psi(\Phi(\mathcal J))$, so that $\operatorname{Supp}_B(X)\subseteq\Phi(\mathcal J)$. For every $(\mathfrak p,F)\in\operatorname{Supp}_B(X)$,
there exists some $Y\in \mathcal J$ such that $(\mathfrak p,F)\in \operatorname{Supp}_B(Y)$. Then by Lemma~\ref{lem3.3} and the identity \[
\bigl(Y\otimes_R^{\mathbf L}\kappa(\mathfrak p)\bigr)
\otimes_{\kappa(\mathfrak p)}\kappa(\mathfrak p)_G
\cong
\bigl(Y\otimes_R^{\mathbf L}\kappa(\mathfrak p)\bigr)
\otimes_R^{\mathbf L}R_G,
\] we have \[
\kappa(\mathfrak p)_F
\in\Loc\left\langle
\bigl(Y\otimes_R^{\mathbf L}\kappa(\mathfrak p)\bigr)
\otimes_{\kappa(\mathfrak p)}\kappa(\mathfrak p)_G:
G\in L_B
\right\rangle\subseteq \mathcal J.
\]
Since $\operatorname{Supp}_B(X)$ is stable, the formula proved first implies that \(X\in\mathcal J\), which concludes the proof.
\end{proof}

Now we can compute the Balmer spectrum of $\Perf(RB)$.

\begin{theorem}\label{thm3.7}
Let $R$ be a commutative Noetherian ring and $B$ a finite left regular band. Then there is a homeomorphism $\Spec R\times L_B\xrightarrow{\sim}\Spc\Perf(RB)$ given by
\[
(\mathfrak p,F)\longmapsto
\mathcal P_{\mathfrak p,F}
=\{X\in\Perf(RB):
\kappa(\mathfrak p)\otimes_R^{\mathbf L}\omega_F(X)=0\}.
\]
\end{theorem}

\begin{proof}
We first show that thick ideals of $\Perf(RB)$ are classified by specialization-closed subsets of $\Spec R\times L_B$, via the mutually inverse assignments
\[
\mathcal I\longmapsto
\bigcup_{X\in\mathcal I}\operatorname{Supp}_B(X),
\qquad
Z\longmapsto
\{X\in\Perf(RB):\operatorname{Supp}_B(X)\subseteq Z\}.
\]

If $\mathcal I$ is a thick ideal of $\Perf(RB)$, then $\Loc\langle\mathcal I\rangle$ is a localizing ideal of $D(RB)$. Indeed, for each $X\in\mathcal I$, the objects $Y$ satisfying
$X\otimes_R^{\mathbf L}Y\in\Loc\langle\mathcal I\rangle$ form a localizing subcategory containing every compact object, and hence all of $D(RB)$.
Thus for each $Y\in D(RB)$, the objects $X$ satisfying the same condition form a localizing subcategory containing $\mathcal I$, and hence $\Loc\langle\mathcal I\rangle$.
The identity
\[
\Loc\langle\mathcal I\rangle\cap\Perf(RB)=\mathcal I
\]
and Theorem~\ref{thm3.6} show that $\mathcal I$ is determined by the union of the tensor supports of its objects. This union is specialization-closed, since every compact
object has closed support by Lemma~\ref{lem3.2}.

Conversely, let $Z\subseteq\Spec R\times L_B$ be specialization-closed. For each $(\mathfrak p,F)\in Z$, choose finitely many generators of
$\mathfrak p$ and let $K(\mathfrak p)$ be the Koszul complex. Then the compact object $R_F\otimes_R K(\mathfrak p)$ has tensor support
\[
V(\mathfrak p)\times\mathord\downarrow F
=\overline{\{(\mathfrak p,F)\}}\subseteq Z.
\]
Thus the thick ideal generated by $R_F\otimes_R K(\mathfrak p)$, as $(\mathfrak p,F)$ ranges over $Z$, has union of tensor supports equal to $Z$. Together with the preceding paragraph, this proves the classification of thick ideals.

Note that every thick ideal is radical. Then $(\Spec R\times L_B,\operatorname{Supp}_B)$ is a  classifying support datum in the sense of \cite[Definition~5.1]{Balmer}.
By \cite[Theorem~5.2]{Balmer}, the canonical map
\[
(\mathfrak p,F)\longmapsto
\{X\in\Perf(RB):
(\mathfrak p,F)\notin\operatorname{Supp}_B(X)\}
=\mathcal P_{\mathfrak p,F}
\]
is a homeomorphism. Under this homeomorphism, $\operatorname{Supp}_B(X)$ corresponds to the Balmer support of $X$ for every $X\in\Perf(RB)$. This concludes the proof.
\end{proof}

At the end of this section, we show that there exists a band algebra $RB$ such that the generalization of Balmer's Nerves of Steel Conjecture to non-rigid setting fails for $\Perf(kB)$.  We refer the reader to \cite{BalmerHomological, BalmerBigSupport} for more details on this conjecture, and to  the recent works \cite{BHR26, Coulembier26, Hyslop} for various counterexamples.
\begin{example}\label{ex3.8}
Let $k$ be a field and let $B=\{1,b\}$ with $b^2=b$. Then \(L_B=\{B,\{b\}\}\) with \( B<\{b\}.\) The element $b$ acts as zero on $k_B$ and as the identity
on $k_{\{b\}}=k$. Write $\mathcal K=\Perf(kB)$ and let $h:\mathcal K\to\operatorname{mod}\mathcal K$ denote the Yoneda embedding. Since $kB\cong k\times k$
as algebras, every object of $\mathcal K$ is a finite direct sum of shifts of $k_B$ and $k_{\{b\}}$. Moreover,
\[
\operatorname{Hom}_{\mathcal K}(k_B[n],k_B[m])
\cong
\operatorname{Hom}_{\mathcal K}
(k_{\{b\}}[n],k_{\{b\}}[m])
\cong
\begin{cases}
k,&n=m,\\
0,&n\ne m,
\end{cases}
\]
and
\[
\operatorname{Hom}_{\mathcal K}
(k_B[n],k_{\{b\}}[m])
=
\operatorname{Hom}_{\mathcal K}
(k_{\{b\}}[n],k_B[m])
=0.
\]
Thus the Freyd envelope $\operatorname{mod}\mathcal K$ is semisimple, with simple objects \(h(k_B[n]), h(k_{\{b\}}[n])\) for each $n\in\mathbb Z$. A proper Serre ideal of $\operatorname{mod}\mathcal K$ cannot contain any $h(k_{\{b\}}[n])$, since tensoring with $h(k_{\{b\}}[-n])$ would give the tensor unit.
On the other hand,
\[
\mathcal S:=\operatorname{add}
\{h(k_B[n]):n\in\mathbb Z\}
\]
is a proper Serre ideal such that every proper Serre  ideal is contained in $\mathcal S$, thus $\mathcal S$ is the unique homological prime and hence
\(\Spc^{\mathrm h}(\mathcal K)=\{\mathcal S\}.\) By Theorem~\ref{thm3.7}, \(\Spc(\mathcal K)\) has two points, thus the natural comparison map
\[
\Spc^{\mathrm h}(\mathcal K)\rightarrow \Spc(\mathcal K)
\]
could not be a homeomorphism.
\end{example}

\section{The tensor abelian spectrum}
\label{sec4}

Let $R$ be a commutative Noetherian ring and $B$ be a finite left regular band. Write $\operatorname{mod}RB$ for the abelian category of finitely generated $RB$-modules.
The ordinary diagonal tensor product on $\operatorname{mod}RB$ is right exact.  A \emph{Serre  ideal} of $\operatorname{mod}RB$ is a Serre subcategory closed under tensoring with arbitrary objects of $\operatorname{mod}RB$. A Serre ideal $\mathcal P$ is \emph{prime} if it is proper and $M\otimes_R N\in\mathcal P$ implies
$M\in\mathcal P$ or $N\in\mathcal P$. Let $\operatorname{Spc}(\operatorname{mod}RB)$ denote the set of these prime ideals, endowed with the topology having
\(\operatorname{supp}(M)=\{\mathcal P\in\operatorname{Spc}(\operatorname{mod}RB):M\notin\mathcal P\}\) as a basis of closed subsets. For further details on tensor abelian geometry, we refer the reader to \cite{BKS07,Xu}.

\begin{theorem}[{\cite[Section~9]{BKS07}}]
\label{thm4.1}
Let $R$ be a commutative Noetherian ring. Then $\operatorname{Spc}(\operatorname{mod}R)\cong \operatorname{Spec}R$. Moreover, Serre ideals of $\operatorname{mod}R$
are classified by specialization-closed subsets
$Z\subseteq\operatorname{Spec}R$ via
\[
Z\longmapsto
\mathcal C_Z
=\{N\in\operatorname{mod}R:
\operatorname{supp}_R N\subseteq Z\}
\]
and the prime Serre ideals of $\operatorname{mod}R$ are precisely \(\{N\in\operatorname{mod}R:N_{\mathfrak p}=0\}\) for $\mathfrak p\in\operatorname{Spec}R$.
\end{theorem}

\begin{theorem}\label{thm4.2}
Let $R$ be a commutative Noetherian ring and $B$ a finite left regular band. Then there is a homeomorphism $\Spec R\times L_B\xrightarrow{\sim}
\operatorname{Spc}(\operatorname{mod}RB)$
given by
\[
(\mathfrak p,F)\longmapsto
\mathcal P^{\mathrm S}_{\mathfrak p,F}
=\{M\in\operatorname{mod}RB:
(\omega_F(M))_{\mathfrak p}=0\}.
\]
\end{theorem}

\begin{proof}
Let $M\in\operatorname{mod}RB$. Then Lemma~\ref{lem3.2} shows that $\operatorname{Supp}_B(M)$ is closed in $\Spec R\times L_B$. Exactness of $\omega_F$ implies that
\(\operatorname{Supp}_B(M)=\operatorname{Supp}_B(M')\cup
\operatorname{Supp}_B(M'')\) for every short exact sequence $0\to M'\to M\to M''\to0$. Strong monoidality of $\omega_F$ and Nakayama's lemma over
each $R_{\mathfrak p}$ show that
\[
\operatorname{Supp}_B(M\otimes_R N)
=\operatorname{Supp}_B(M)\cap\operatorname{Supp}_B(N).
\]
In particular, for every specialization-closed subset $Z\subseteq\Spec R\times L_B$, the subcategory $\{M\in\operatorname{mod}RB:\operatorname{Supp}_B(M)\subseteq Z\}$ is a Serre ideal.

We now claim that Serre ideals of $\operatorname{mod}RB$ are classified by specialization-closed subsets of $\Spec R\times L_B$, via the mutually inverse assignments
\[
\mathcal I\longmapsto
\bigcup_{M\in\mathcal I}\operatorname{Supp}_B(M),
\qquad
Z\longmapsto
\{M\in\operatorname{mod}RB:
\operatorname{Supp}_B(M)\subseteq Z\}.
\]

Let $\mathcal I$ be a Serre ideal. For $F\in L_B$, write
\[
\mathcal C_F:=\{N\in\operatorname{mod}R:
R_F\otimes_R N\in\mathcal I\}.
\]
The functor $N\mapsto R_F\otimes_R N$ is exact, so $\mathcal C_F$ is a Serre ideal. By Theorem~\ref{thm4.1}, there is a unique specialization-closed subset $Z_F\subseteq\Spec R$ such that
\[
\mathcal C_F
=\{N\in\operatorname{mod}R:
\operatorname{supp}_R N\subseteq Z_F\}.
\]
Note that, for $G\leq F$, the identity $R_G\otimes_R R_F\cong R_G$ implies $\mathcal C_F\subseteq\mathcal C_G$, and hence $Z_F\subseteq Z_G$.
Since each $Z_F$ is specialization-closed, so is $\bigcup_{F\in L_B}Z_F\times\{F\}$. The finite  filtration of $M$ and Lemma~\ref{lem3.2} then give
\[
\mathcal I
=\{M\in\operatorname{mod}RB:
\operatorname{Supp}_B(M)\subseteq
\bigcup_{F\in L_B}Z_F\times\{F\}\}.
\]
For every $(\mathfrak p,F)\in \bigcup_{F\in L_B}Z_F\times\{F\}$, specialization closure gives $V(\mathfrak p)\subseteq Z_F$. Hence
$R_F\otimes_R R/\mathfrak p\in\mathcal I$, and \(\operatorname{Supp}_B(R_F\otimes_R R/\mathfrak p)=V(\mathfrak p)\times\mathord\downarrow F.\) Thus we have
$\bigcup_{M\in\mathcal I}\operatorname{Supp}_B(M)=\bigcup_{F\in L_B}Z_F\times\{F\}$.
Conversely, let $Z\subseteq\Spec R\times L_B$ be specialization-closed. For each $(\mathfrak p,F)\in Z$, the module $R_F\otimes_R R/\mathfrak p$ has support
$V(\mathfrak p)\times\mathord\downarrow F\subseteq Z$, so it belongs to the Serre ideal $\{M:\operatorname{Supp}_B(M)\subseteq Z\}$.
Then the tensor support of $\{M\in\operatorname{mod}RB:\operatorname{Supp}_B(M)\subseteq Z\}$ is exactly $Z$. This proves  the claim.

Under the classification above, $(\Spec R\times L_B,\operatorname{Supp}_B)$ is a classifying support datum on $\operatorname{mod}RB$. The conclusion then follows immediately from \cite[Corollary~6.2]{BKS07}.
\end{proof}

\begin{corollary}\label{cor4.3}
Let $R$ be a commutative Noetherian ring and $B$ a finite left regular band. Then there is a canonical homeomorphism
\[
\operatorname{Spc}(\operatorname{mod}RB)\longrightarrow\Spc(\Perf(RB)),
\qquad \mathcal P\longmapsto\{C\in\Perf(RB):H^i(C)\in\mathcal P\text{ for all }i\in\mathbb Z\}.
\]
\end{corollary}

\section{The telescope conjecture}\label{sec5}

In this section, we classify all smashing ideals of $D(RB)$, where $R$ is a commutative Noetherian ring and $B$ is a finite left regular band. As a consequence, we prove that the telescope conjecture holds for $D(RB)$. When $B$ is trivial, this recovers Neeman's theorem.
\begin{notation}
For a stable subset $W\subseteq\Spec R\times L_B$ and $F\in L_B$, write
\[
W_F=\{\mathfrak p\in\Spec R:(\mathfrak p,F)\in W\}.
\]
For $V\subseteq\Spec R$, write
\[
D_V(R)=\{N\in D(R):\supp_R(N)\subseteq V\}.
\]
\end{notation}

\begin{lemma}\label{lem5.2}
Let $W\subseteq\Spec R\times L_B$ be stable and let $\Phi$, $\Psi$ be as in Theorem~\ref{thm3.6}. If $\Psi(W)$ is a smashing ideal, then $D_{W_F}(R)$ is a smashing ideal of $D(R)$ for every $F\in L_B$.
\end{lemma}

\begin{proof}
Note that $W_F\subseteq W_G$ whenever $G\leq F$, and Lemma~\ref{lem3.2} gives
\[
\Psi(W)=
\{X\in D(RB):
\supp_R(e_FX)\subseteq W_F
\text{ for every }F\in L_B\}.
\]
Thus we have $e_FX\in D_{W_F}(R)$ for every $X\in\Psi(W)$. Conversely, for $N\in D_{W_F}(R)$,
Lemma~\ref{lem2.8} gives
\[
e_G(R_F\otimes_R N)\cong
\begin{cases}
N,&G=F,\\
0,&G\ne F.
\end{cases}
\]
Hence $R_F\otimes_R N\in\Psi(W)$ and $e_F(R_F\otimes_R N)\cong N$. Thus we have $e_F(\Psi(W))=D_{W_F}(R)$.

Since \(e_F(-)\cong e_F(RB)\otimes_{RB}^{\mathbf L}-:D(RB)\longrightarrow D(R)\) with $e_F(RB)$ finitely generated projective over $R$, it has a coproduct-preserving right adjoint
\(\RHom_R(e_F(RB),-).\) Then the identity $e_F(\Psi(W))=D_{W_F}(R)$ and the derived adjunction give
\[
N\in D_{W_F}(R)^\perp
\quad\Longleftrightarrow\quad
\RHom_R(e_F(RB),N)\in\Psi(W)^\perp.
\]
Since $\Psi(W)$ is smashing, it follows that $D_{W_F}(R)$ is also smashing. Finally, recall that every localizing subcategory of $D(R)$ is a tensor ideal. Thus $D_{W_F}(R)$ is
a smashing ideal of $D(R)$.
\end{proof}

\begin{theorem}\label{thm5.3}
Let $R$ be a commutative Noetherian ring and $B$ a finite left regular band. Suppose $W\subseteq\Spec R\times L_B$ is stable. Then the following  are equivalent:
\begin{enumerate}[label=\textup{(\roman*)}]
\item $\Psi(W)$ is smashing;
\item $W$ is specialization-closed in $\Spec R\times L_B$;
\item $\Psi(W)$ is generated by compact objects
      as a localizing subcategory.
\end{enumerate}
In particular, $D(RB)$ satisfies the telescope conjecture.
\end{theorem}

\begin{proof}
Assume (i).
By Lemma~\ref{lem5.2}, $D_{W_F}(R)$ is smashing for every $F\in L_B$. Neeman's smashing classification \cite[Theorem~3.3]{Neeman} implies that every $W_F$ is specialization-closed.
If $(\mathfrak p,F)\in W$, $\mathfrak p\subseteq\mathfrak q$, and $G\leq F$, then \(\mathfrak q\in W_F\subseteq W_G.\) This proves (ii).

Assume (ii). Then each $W_F$ is specialization-closed, and the compact object $R_F\otimes_R K(\mathfrak p)$ satisfies
\[
\operatorname{Supp}_B(R_F\otimes_R K(\mathfrak p))
=V(\mathfrak p)\times\mathord\downarrow F
\subseteq W
\]
whenever $(\mathfrak p,F)\in W.$ Thus  we have  $R_F\otimes_R K(\mathfrak p)\in\Psi(W)$ for $(\mathfrak p,F)\in W$. On the other hand, by Neeman's classification, we have
\[
D_{W_F}(R)=
\Loc\langle K(\mathfrak p):
\mathfrak p\in W_F\rangle.
\]
It follows that for $N\in D_{W_F}(R)$, we have $R_F\otimes_R N\in \Loc\langle R_F\otimes_R K(\mathfrak p):\mathfrak p\in W_F\rangle$. Now let $X\in\Psi(W)$.
For every $F\in L_B$, one has $e_FX\in D_{W_F}(R)$. Hence every quotient $R_F\otimes_R e_FX$ in the finite filtration of Lemma~\ref{lem2.11} belongs to
\[
\Loc\langle R_G\otimes_R K(\mathfrak p):
(\mathfrak p,G)\in W\rangle.
\]
The triangles associated with this filtration show that $X$ belongs to the same subcategory. This proves  (iii).

The implication (iii)$\Rightarrow$(i) is standard; see, for example, \cite[Theorem~1.7]{HallRydh17}. This concludes the proof.
\end{proof}

\begin{corollary}
The restrictions of $\Phi$ and $\Psi$ give mutually inverse lattice isomorphisms
\[
\left\{\text{smashing ideals of }D(RB)\right\}
\underset{\Psi}{\overset{\Phi}{\rightleftarrows}}
\left\{W\subseteq\Spec R\times L_B:
W\text{ is specialization-closed}\right\}.
\]
\end{corollary}

\begin{proof}
This follows from Theorems~\ref{thm3.6} and~\ref{thm5.3}.
\end{proof}

\section{Examples}
\label{sec6}

In this section, we consider examples of band algebras arising from matroids, oriented matroids, finite simple graphs, and finite acyclic quivers. The case of the path algebra $RQ$ is of particular interest, since $\Perf(RQ)$ admits a band tensor product whose associated Balmer spectrum generally differs from that induced by the vertexwise tensor product considered by Sabatini \cite{Sabatini}. For the convenience of readers who are unfamiliar with these combinatorial concepts, we spell out the relevant definitions as much as possible. Theorem~\ref{thm3.7} and Theorem~\ref{thm5.3} apply to each of the examples.  In particular,  the telescope conjecture holds for all band algebras appearing in this section.

\subsection{Ordinary matroids and independent words}
A \emph{finite matroid} is a pair $M=(E,\mathcal I)$, where $E$ is a finite set and $\mathcal I\subseteq 2^E$ is a family of subsets, called \emph{independent sets},
satisfying:
\begin{enumerate}
\item $\varnothing\in\mathcal I$;
\item if $I\in\mathcal I$ and $J\subseteq I$, then
      $J\in\mathcal I$;
\item if $I,J\in\mathcal I$ and $|I|<|J|$, then there
      exists $e\in J\setminus I$ such that
      $I\cup\{e\}\in\mathcal I$.
\end{enumerate}
For $S\subseteq E$, its \emph{rank} is defined by
\(\operatorname{rk}_M(S)
:=\max\{|I|:I\subseteq S,\ I\in\mathcal I\}.\)
The \emph{closure} of $S$ is
\(
\operatorname{cl}_M(S)
:=\{e\in E:
\operatorname{rk}_M(S\cup\{e\})
=\operatorname{rk}_M(S)\}.\)
A subset $F\subseteq E$ is called a \emph{flat} if
$\operatorname{cl}_M(F)=F$.
The set \(\operatorname{Flats}(M):=\{F\subseteq E:\operatorname{cl}_M(F)=F\}\) ordered by inclusion, is a lattice with
\[
F\wedge G=F\cap G,\qquad
F\vee G=\operatorname{cl}_M(F\cup G).
\]
A word $w=e_1\cdots e_r$ is independent if its letters are distinct and $\{e_1,\ldots,e_r\}\in\mathcal I$.  For any word $w$, let $g(w)$ be the word obtained by
reading from left to right and retaining a letter exactly when it is not in the closure of the previously retained letters. Now define
\[
 B(M)=\{\text{independent words on }E\},
 \qquad u\cdot v=g(uv).
\]
By \cite[Section~3.4 and Example~3.5]{MSS}, $B(M)$ is a finite left regular band and there is a canonical lattice isomorphism
\[
L_{B(M)}\xrightarrow{\sim}\operatorname{Flats}(M),
\qquad
B(M)w\longmapsto
\operatorname{cl}_M\{\text{letters of }w\}.
\]
This identifies the topology on $L_{B(M)}$ with the topology on $\operatorname{Flats}(M)$ whose closed subsets are precisely the downward-closed subsets. Then the following result is immediate from Theorem~\ref{thm3.7}.

\begin{proposition}\label{prop6.1}
Let $R$ be a commutative Noetherian ring and $M$ a finite matroid. Then  we have \(\Spc\Perf(RB(M))\cong\Spec R\times\operatorname{Flats}(M).\)
\end{proposition}

\begin{example}
For $1\leq r\leq n$, the \emph{uniform matroid} $U_{r,n}$ is the matroid on $E=\{1,\ldots,n\}$ whose independent sets are precisely the subsets of cardinality at most $r$:
\[
\mathcal I=\{I\subseteq E:|I|\leq r\}.
\]The flats are the subsets of cardinality less than $r$, together with the entire ground set.  Thus, over a field, the spectrum has
\[
 1+\sum_{j=0}^{r-1}\binom{n}{j}
\]
points. For example, the prime ideals of $\Perf(kB(U_{2,3}))$ correspond bijectively to the flats $\varnothing$, $\{1\}$, $\{2\}$, $\{3\}$, and $E=\{1,2,3\}$.
\end{example}

\begin{example}\label{ex6.3}
Let $q$ be a prime power and $n\geq 0$. Let $M_{n,q}$ be the vector matroid on $\mathbb F_q^n$, whose independent sets are the linearly independent subsets. Then
$B_{n,q}:=B(M_{n,q})$ consists of  linearly independent tuples. This is the band of \cite[Example~3.4]{MSS}.

The flats of $M_{n,q}$ are precisely the linear subspaces of $\mathbb F_q^n$. Hence Proposition~\ref{prop6.1} gives
\[
\Spc\Perf(RB_{n,q})
\cong
\Spec R\times\operatorname{Sub}(\mathbb F_q^n),
\]
where the subspaces are ordered by inclusion.
\end{example}

\subsection{Oriented matroids}
There is a second band associated with a matroid.  A \emph{sign vector} on $E$ is a function $X:E\to\{0,+,-\}$, denoted by $X=(X_e)_{e\in E}$.
Let $\{0,+,-\}^{E}$ be the set of all sign vectors on $E$ and let $-X$ be obtained from $X$ by interchanging $+$ and $-$. For sign vectors $X,Y$, define their \emph{composition} by
\[
 (X\circ Y)_e=
 \begin{cases}
 X_e,&X_e\neq 0,\\
 Y_e,&X_e=0,
 \end{cases}
\]
and their \emph{separation set} by \(S(X,Y)=\{e\in E:X_e=-Y_e\neq 0\}.\) An \emph{oriented matroid} on $E$ is a pair $\mathcal M=(E,\mathcal V^*)$, where
$\mathcal V^*\subseteq\{0,+,-\}^{E}$ satisfies the covector axioms:
\begin{enumerate}
\item $\mathbf 0\in\mathcal V^*$;
\item $X\in\mathcal V^*$ implies $-X\in\mathcal V^*$;
\item $X,Y\in\mathcal V^*$ implies $X\circ Y\in\mathcal V^*$;
\item if $X,Y\in\mathcal V^*$ and $e\in S(X,Y)$, then there
exists $Z\in\mathcal V^*$ such that $Z_e=0$ and \(Z_f=(X\circ Y)_f=(Y\circ X)_f\) for every \(f\notin S(X,Y).\)
\end{enumerate}
The elements of $\mathcal V^*$ are called \emph{covectors}. Their zero sets
\[
 Z(X)=\{e\in E:X_e=0\},\qquad X\in\mathcal V^*,
\]
are precisely the flats of an ordinary matroid on $E$, denoted $\underline{\mathcal M}$ and called the \emph{underlying matroid} of $\mathcal M$. For more details on oriented matroids, see \cite{HSX25,NPS02}.

We denote by $B_{\mathrm{cov}}(\mathcal M)$ the monoid  $\mathcal V^*$, whose multiplication is given by composition, and whose identity is $\mathbf 0$. By \cite[Section~3.2.1]{MSS}, $B_{\mathrm{cov}}(\mathcal M)$ is a finite left regular band.

\begin{proposition}\label{prop6.4}
Let $R$ be a commutative Noetherian ring and $\mathcal M$ an oriented matroid. Then we have \(\Spc\Perf\bigl(RB_{\mathrm{cov}}(\mathcal M)\bigr)\cong\Spec R\times\operatorname{Flats}(\underline{\mathcal M})^{\op}.\)
\end{proposition}

\begin{proof}

Under our join convention, the zero-set map identifies the support lattice of $B_{\mathrm{cov}}(\mathcal M)$ with $\operatorname{Flats}(\underline{\mathcal M})^{\op}$.
Indeed, by definition, we have \(X\circ Y=X\) if and only if \(Z(X)\subseteq Z(Y).\) Thus two covectors have the same support exactly when they
have the same zero set, and the support order corresponds to reverse inclusion of zero sets. Then the homeomorphism  follows from Theorem~\ref{thm3.7}.
\end{proof}

\subsection{Free partially commutative bands of graphs}

Let $\Gamma=(V,E_\Gamma)$ be a finite simple graph, that is, a finite undirected graph with no loops or multiple edges. Write
\[
F(V)=\{v_1\cdots v_r:0\le r\le |V|,\ v_i\in V,\
v_i\ne v_j\text{ whenever }i\ne j\}.
\]
For $u=u_1\cdots u_r$ and $v=v_1\cdots v_s$, define
\[
u\cdot v:=u_1\cdots u_rv_{a_1}\cdots v_{a_s},\] where $v_{a_i}=\emptyset$ if $v_i=u_j$ for some $j$, and $v_{a_i}=v_i$ otherwise. Define
\[
 B(\Gamma)=F(V)/\langle xy=yx:(x,y)\in E_\Gamma\rangle,
\]
where the quotient is by the monoid congruence generated by these relations.  This is a finite left regular band with support lattice \(2^V\), by
\cite[Section~3.6]{MSS}. Then the following result is immediate.

\begin{proposition}\label{prop6.5}
Let $R$ be a commutative Noetherian ring and $\Gamma=(V,E_\Gamma)$  a finite simple graph. Then we have \(\Spc\Perf(RB(\Gamma))\cong\Spec R\times(2^V).\)
\end{proposition}

\begin{remark}
Let $\Gamma=(V,E_\Gamma)$ be a  finite simple graph. Then by Theorem~\ref{thm3.6} and Proposition~\ref{prop6.5}, the classification of localizing ideals of $D(RB(\Gamma))$  depends only on the vertex set $V$, not on the edge set $E_\Gamma$.
\end{remark}

\subsection{Acyclic quivers}

Let $R$ be a commutative Noetherian ring and $Q=(Q_0,Q_1)$ a finite acyclic quiver. For each vertex $v\in Q_0$, write $\varepsilon_v$ for the stationary path of length 0 with source
and target equal to $v$ . A  path $p=\alpha_1\cdots\alpha_m$ is a sequence of arrows satisfying $t(\alpha_i)=s(\alpha_{i+1})$ for
$1\le i<m$; its source and target are $s(p)=s(\alpha_1)$ and $t(p)=t(\alpha_m)$.

For  paths $p=\alpha_1\cdots\alpha_m$ and $q=\beta_1\cdots\beta_n$, multiplication in the path algebra $RQ$ is given by
\[
pq=
\begin{cases}
\alpha_1\cdots\alpha_m\beta_1\cdots\beta_n,
  &t(p)=s(q),\\
0, &t(p)\ne s(q).
\end{cases}
\]
In particular, $\varepsilon_{s(\alpha)}\alpha=\alpha=\alpha\varepsilon_{t(\alpha)}$. Choose a total order $\prec$ on $Q_0$ such that $s(\alpha)\prec t(\alpha)$ for every arrow $\alpha$. For a  path $p=\alpha_1\cdots\alpha_m$, define in $RQ$
\[
\ell(p)=\sum_{u\succeq t(p)}\varepsilon_u
       +\sum_{i=1}^{m}\alpha_i\cdots\alpha_m
\]
and write $B_Q^\prec:=\{\ell(p):p\text{ is a path of }Q\}$. By \cite[Theorem~3.10]{MSS}, $B_Q^\prec$ is a finite left regular band such that $RB_Q^\prec=RQ$.

Through this algebra isomorphism, there exists a  diagonal band tensor product on $\Perf(RQ)$, denoted by  $\otimes_{\mathrm{band}}^{\mathbf L}$.
Recall that the usual tensor product $\otimes_{\mathrm{vert}}^{\mathbf L}$ of quiver representations is defined vertexwise, see \cite{Sabatini}. The following result shows that these two tensor structures can  have non-homeomorphic Balmer spectra.

\begin{theorem}\label{thm6.7}
Let $R$ be a commutative Noetherian ring and $Q=(Q_0,Q_1)$ a finite acyclic quiver. Then the  triangulated category $\Perf(RQ)$ carries tensor products with spectra
\begin{align*}
 \Spc(\Perf(RQ),\otimes_{\mathrm{band}}^{\mathbf L})
   &\cong\Spec R\times(Q_0,\prec),\\
 \Spc(\Perf(RQ),\otimes_{\mathrm{vert}}^{\mathbf L})
   &\cong\Spec R\times Q_0^{\mathrm{disc}}.
\end{align*}
Moreover, $D(RQ)$ satisfies the telescope conjecture for both tensor products.
\end{theorem}

\begin{proof}
As in the proof of \cite[Theorem~3.10]{MSS}, $(Q_0,\prec)$ is isomorphic to the support lattice of $B_Q^\prec$, so the first homeomorphism follows from the identity $RB_Q^\prec=RQ$ and Theorem~\ref{thm3.7}.  The second homeomorphism is \cite[Theorem~2.2]{Sabatini}.  By Theorem~\ref{thm5.3}, $(D(RQ),\otimes_{\mathrm{band}}^{\mathbf L})$ satisfies the telescope conjecture. To show $(D(RQ),\otimes_{\mathrm{vert}}^{\mathbf L})$ satisfies the telescope conjecture, it suffices to apply \cite[Theorem~4.4]{Sabatini} to stable tensor
$t$-structures. Indeed, the right orthogonal of a smashing localizing subcategory is triangulated and closed under coproducts, hence under directed derived homotopy colimits, so this stable case is homotopically smashing.
\end{proof}

\begin{example}\label{ex6.8}
Let $k$ be a field and $Q=(1\xrightarrow{\alpha}2)$ . Then  $\Spc(\Perf(kQ),\otimes_{\mathrm{band}}^{\mathbf L})$ is a two-point chain, while $\Spc(\Perf(kQ),\otimes_{\mathrm{vert}}^{\mathbf L})$  is a discrete two-point space.
\end{example}

\vspace{4mm}

\textbf{Peng Xu}\\
School of Mathematics, Nanjing University,
Nanjing 210093, P. R. China;\\
E-mail: \textsf{602023210017@smail.nju.edu.cn}\\[1mm]

\end{document}